\documentclass[reqno]{amsart}
\usepackage{amssymb, latexsym,amsmath,amsthm,enumerate}
\usepackage[mathscr]{eucal}
\usepackage{framed,color,graphicx}
\usepackage{mathrsfs}
\usepackage[all]{xy} 
\usepackage{tikz}
\usetikzlibrary{positioning}

\usepackage{appendix}
\DeclareSymbolFont{bbold}{U}{bbold}{m}{n}
\DeclareSymbolFontAlphabet{\mathbbold}{bbold}

\theoremstyle{plain} 
\newtheorem{thm}{Theorem}[section] 
\newtheorem{lem}[thm]{Lemma}
\newtheorem{cor}[thm]{Corollary}

\newtheorem{problem}[thm]{Problem}

\usepackage{thmtools}
\usepackage{thm-restate}

\theoremstyle{definition}
\newtheorem{defn}[thm]{Definition}

\newtheorem{remark}[thm]{Remark}

\newcommand{\up}[1]{\textup{#1}}

\newcommand{\dom}{\operatorname{dom}}
\newcommand{\ran}{\operatorname{ran}}

\newcommand{\KL}{\mathsf{K}_L}
\newcommand{\KLn}[1]{\mathsf{K}_{L,#1}}
\newcommand{\KR}{\mathsf{K}_R}
\newcommand{\KRn}[1]{\mathsf{K}_{R,#1}}
\newcommand{\DD}{\mathsf{D}}
\newcommand{\RR}{\mathsf{R}}
\newcommand{\Ad}{\mathsf{A_d}}
\newcommand{\Ar}{\mathsf{A_r}}

\newcommand{\II}{\mathscr{I}}
\newcommand{\uu}{\mathsf{u}}

\newcommand{\Rel}{{\sf Rel}}

\newcommand{\join}{\mathbin{\ooalign{\hfil$\sqcup$\hfil\cr\hfil\raise0.28ex\hbox{\tiny\textup{+}}\hfil\cr}}}

\newcommand{\into}{\rightarrowtail}
\newcommand{\comp}{\mathbin{;}}
\newcommand{\conv}{\smile}

\def\c#1{{\mathcal #1}}
\def\nb#1{{$\bullet$ \marginpar{$\bullet$ #1}}}
\def\set#1{{\{#1\} }}
\def\P{{\bf P}}
\def\lrangle#1{\langle#1\rangle}
\def\E{{\bf E}}

\def\G{{\bf G}}

\makeatletter
\newcommand*\bigcdot{\mathpalette\bigcdot@{.75}}
\newcommand*\bigcdot@[2]{\mathbin{\vcenter{\hbox{\scalebox{#2}{$\m@th#1\bullet$}}}}}
\makeatother

\begin{document}

\title[]{Minimal signatures with undecidability of representability by binary relations}
\author{Robin Hirsch}
\address{Department of Computer Science, University College London,
Gower Street, London WC1E 6BT
United Kingdom} \email{r.hirsch@ucl.ac.uk}
\author{Marcel Jackson}
\address{Department of Mathematical and Physical Sciences, La Trobe University, Victoria  3086,
Australia} \email{m.g.Jackson@latrobe.edu.au}
\author{Ja\v{s} \v{S}emrl}
\address{Division of Computing, University of the West of Scotland, London,
UK} \email{jas.semrl@uws.ac.uk}

%\subjclass[2010]{Primary: 68Q17, 20M07; Secondary: 03C05, 08B99, 08C15}

%%Robin's \subjclass[2010]{Primary: 06F05, 03E20, 03D25}
%\thanks{The author was supported by ARC Discovery Project DP1094578 and ARC Future Fellowship FT120100666}

\begin{abstract}
A semigroup of binary relations (under composition) on a set $X$ is \emph{complemented} if it is closed under the taking of complements within $X\times X$.  We resolve a 1991 problem of Boris Schein by showing that the class of finite unary semigroups that are representable as complemented semigroups of binary relations is undecidable, so composition with complementation forms a minimal subsignature of Tarski's relation algebra signature that has undecidability of representability.  In addition we prove similar results for semigroups of binary relations endowed with unary operations returning the kernel and cokernel of a relation. We generalise to signatures which may include arbitrary, definable operations and provide a chain of weaker and weaker signatures, each definable in the previous signature, each having undecidability of representability, but whose limit signature includes composition only, which  corresponds to the well known, decidable and finitely axiomatised variety of semigroups.  All these results are also proved for  representability as binary relations over a finite set.
\end{abstract}

\maketitle
\section{Introduction} A \emph{representation} is an isomorphism from an algebra with a signature of abstract operations to an algebra of binary relations with specified, natural, set-theoretic operations.   
Every semigroup  $(A, \comp )$ may be represented as binary relations over some set under the operation of relational composition.  
All signatures considered here will include relational composition, which we denote by $\comp $ as this is standard in one of the main contexts of our contributions, and avoids confusion with the many other operations that are considered.   
The signature of relation algebra is $\{0, 1, -, +, 1', {}^\conv, \comp\}$,\footnote{Following a suggestion of Andr\'eka we adopt the following convention for the ordering of functions and relations: Booleans $0, 1, -, +, \cdot, \leq$ first, then non-Booleans, with both of these two parts put constants and functions first in increasing order of arity, then relations.} a representation is an isomorphism from an abstract algebra of this signature to an algebra of binary relations on some arbitrary set $X$, 
with relational composition, union, complement relative to a top binary relation, relational converse,   the empty set,  the top and the identity over $X$.  
A relation algebra is an algebra of this signature obeying certain axioms specified by Tarski.  
Not every relation algebra is representable \cite{lynI} and  there can be no algorithm that will determine whether a given finite algebra  has such a representation \cite{hirhod}.  Thus the \emph{representation problem} is undecidable, for the signature of relation algebra.  

As striking as this undecidability result may be, it was known as early as the 1940s that once the signature $\tau$ of consideration is expanded beyond $\{\comp\}$ alone,  the representation problem  becomes challenging.  
Tarski had already failed to fully classify the representable algebras in the full signature  of relation algebra, despite significant insights and development~\cite{tar}.  
By the 1950s, Lyndon~\cite{lynI,lynII} had provided a  complicated, highly technical and infinite axiomatisation for the abstract representable algebras, and shown that at least some further laws were necessary beyond those initially proposed by Tarski.  
By the 1960s, Monk \cite{mon} showed that no finite axiomatisation is possible, and  signatures weaker than the full relation algebra signature had started to achieve more attention. These weaker signatures may include any (proper) subset of 
$\set{0, 1, -, +, 1', {}^\conv, \comp}$, but may also include operators such as $\cdot, \DD, \RR$, etc.\ 
\emph{definable} by relation algebra terms,  represented as $\cap$, the restriction of the identity to the domain and range,  
etc.~respectively, and may also include relations such as $\leq$, definable by relation algebra equations, represented as~$\subseteq$.  
See Section~\ref{sec:operations} for several examples of such functions and relations, along with their definitions. Signatures are ordered by definability.  
J\'onsson~\cite{jon} considered the  signature $\{\cdot,{}^\conv, \comp\}$, Schein, Bredikhin \cite{bre,bresch,sch} and a number of other authors (see surveys \cite{sch70,sch91}) had considered variations such as $\{{}^\conv, \comp\}, \; \{\leq,{}^\conv, \comp\},\; \{\cdot, \comp\}$ and many others.  
While each signature can carry a different challenge, anecdotally the general theme is that classifications are difficult and complicated: no finite axiomatisation 
is  possible, and the majority fail to have the finite representation property (henceforth: FRP, meaning that finite representable algebras are representable as binary relations on a finite set).  
The main known exceptions to this are the following representation classes. 
\begin{itemize}
\item $\{\leq, \comp\}$ (where $\leq$ denotes a partial order intended to be represented as  containment), which Zareckii \cite{zar} showed is just the class of ordered semigroups, and all finite ordered semigroups are representable over finite sets.
\item $\{\cdot, \comp\}$ which Bredikhin and Schein \cite{bresch} showed is just the class of semigroups with an additional semilattice operation whose underlying order is preserved by composition, though the FRP fails in general~\cite{neu}.  Later Bredikhin~\cite{bre98} showed that a few extra axioms could capture constants representing the universal relation and the empty relation. 
\item $\{\leq, \DD,\RR,{}^\conv,\comp\}$, where $\DD$ and $\RR$ denote the \emph{domain} and \emph{range} operations respectively: $\DD(r)=\{(x,x)\mid  \exists y\ (x,y)\in r\}$, and dually for the $\RR$ operation.  Bredkhin~\cite{bre} and later Hirsch and Mikulas~\cite{hirmik} showed that some simple axioms are sufficient to ensure representability and this remains true when the identity $1'$ and $0$ are  added to the signature.
The FRP follows from the methodology of~\cite{hirmik}.
\end{itemize} 
The signature $\{\leq, \DD,\RR,{}^\conv,\comp\}$ is sometimes referred to as the ``oasis'' signature, as nearby signatures, richer or weaker,   exhibit challenging and complex behaviour (though $\leq$ can be dropped without the loss of the finite representation property~\cite{sem21}).
%Two exceptional cases are the Bredikhin-Schein signature $\{\comp,\cdot\}$, where there is a simple classification for representability over an arbitrary set \cite{bresch}, but no known classification for representation of an arbitrary finite set, and the ``oasis'' signature $\{\comp,\DD,\RR,{}^\conv,\leq\}$, where there is an elegant finite axiomatisation \cite{bre} and all finite representable algebras are finitely representable \cite{hirmik}, yet nearby signatures, richer or weaker, return to exhibit challenging and complex behaviour.   Here $\DD$ and $\RR$ denote the domain and range operations respectively: $\DD(r)=\{(x,x)\mid \exists y\ (x,y)\in r\}$, and dually for $\RR$.

Undecidability of representability (or of finite representability) remains perhaps the most extreme negative property, implying both non-finite axiomatisability and failure of the finite representation property.  
Following the Hirsch-Hodkinson result for the Tarski signature, there have been a number of efforts to identify precisely how weak the signatures with this kind of negative behaviour can be.  
Hirsch and Jackson~\cite{hirjac} made a significant advance here with the introduction of a new method for signatures avoiding converse, and showed that signatures such as $\{+,\cdot,1', \comp\}$, $\{\backslash,1',\comp\}$, $\{\Rightarrow,1', \comp\}$ and $\{-,\leq,1', \comp\}$ exhibit undecidability of representability.  These signatures are also shown to exhibit undecidability of finite representability, a property that is currently still open for the original Tarski signature.
Neuzerling~\cite{neu} later showed that the Hirsch-Jackson construction continues to work for the weaker signatures $\{+,\cdot, \comp\}$ and $\{-,\leq, \comp\}$, where $1'$ is omitted, and more recently Lewis-Smith and \v{S}emrl~\cite{L-SS} have shown the same for $\{\Rightarrow,\comp\}$.  The idea has been further employed by Hirsch and \v{S}emrl \cite{hirsem} in the context of the so-called ``demonic composition''.  In all cases, these negative results hold for any richer signature provided converse is omited  (though more exotic operations such as reflexive-transitive closure operation can be added as well~\cite{hirjac}).  The tiling method of Hirsch and Hodkinson~\cite{hirhod} also extends downward to establish undecidability of representability for any subsignature of Tarski's signature that contains J\'onsson's~\cite{jon} signature $\{\cdot,{}^\conv, \comp\}$, though the undecidability of finite representability remains open for all of these.

\subsection{Results} The present paper provides three results concerning the boundaries of decidability and undecidability for representability and finite representability.  The abbreviation DR will mean \emph{decidability of representability} and DFR will mean \emph{decidability of finite representability}.  We use UDR and UDFR for the failure of these properties, that is, for undecidability of representability and finite representability.

%There are many natural term operations in the signature, definable but not explicitly included in the relation algebra signature, so here we consider \emph{term-definable} signatures where each constant or operation included in the signature is defined by a relation algebra term, we may also allow relation symbols definable by equations between relation algebra terms, see section~\ref{sec:operations} below for several examples and their definitions.  Signatures are ordered by definability.
From amongst Tarski's original~\cite{tar} signature of operations $\{0, 1, -, +,\cdot,1', {}^\conv,\comp\}$, we identify a minimal subsignature that has UDR, and also UDFR.
\begin{restatable}{thm}{thmmaincomplement}\label{thm:maincomplement}
The following two problems are undecidable\up:  Is a given finite  $\set{-, \comp }$-structure   $\set{-, \comp }$-representable as binary relations over some \up[finite\up] set $X$  \up(with complementation relative to  $X\times X$ and composition\up)?
\end{restatable}
The theorem covers two problems because here and below the parenthesised ``[finite]'' may be included or omitted, consistently throughout. 
\begin{cor}
    Let $S$ be any signature containing $\set{-, \comp}$ and definable by $\{0, 1,-,\allowbreak +, 1', \comp\}$, the signature of Boolean monoids.  The following problem is undecidable: Is a given  finite $S$-structure $S$-representable as binary relations \up[over a finite set\up]?
\end{cor}
An exploration of the signature $\{-,\comp\}$ in terms of possible characterisation %\nb{characterisation means axiomatisation?  MJ: probably, but the footnote covers the bases on this I think.}
is given as an open problem in Schein's~1991 survey~\cite{sch91}, and  the UDR, UDFR and finite axiomatisability is Problem~3.10 of Neuzerling~\cite{neu}; these are solved\footnote{In terms of Schein's problem: technically, recursive axiomatisations do exist and can be systematically constructed via what Schein calls the Fundamental Theorem of Relation Algebras~\cite{sch70}, but our results show that such an axiomatisation can never be simple enough to be recursively verifiable on finite algebras.} by Theorem~\ref{thm:maincomplement}.

 However, Theorem~\ref{thm:maincomplement} does not cover signatures including functions definable using ${}^\conv$.
Two examples of such  functions are the unary $\KL, \KR$, defined in relation algebra by $\KL(r)=r\comp r^\conv,\;\KR(r)=r^\conv\comp r$.
The next theorem shows that the signature $\{\KL,\KR, \comp \}$ has UDR and UDFR, even though it is a term reduct of the oasis signature $\{\leq, \DD, \RR, {}^\conv,\comp \}$ with the FRP and with DR, further reinforcing  the ``oasis'' character.  
\begin{restatable}{thm}{thmmainkernel}\label{thm:mainkernel}
Let $\tau$ be a signature containing $\{\KL,\KR, \comp \}$ and contained within either  $\{0, 1, -, \leq,0', 1', \DD,\RR,\Ad,\Ar,\KL,\KR,\join,\comp \}$ or $\{0,1,  \cdot, 0', 1',\DD,\RR,\Ad,\Ar,\KL,\KR,\allowbreak \join,\comp\}$.  Then  $\tau$ has the UDR and UFDR property.
\end{restatable}
The functions  $\join, \Ad$ etc. will be defined in Section~\ref{sec:operations}.
We are able to recast the key idea of Theorem \ref{thm:mainkernel} in purely abstract terms.  A \emph{unary semigroup} is an algebra $( S,*,{}')$ consisting of an associative binary operation $*$ and a unary operation $'$: there is no restriction on the properties of the unary operation $'$.  Most well studied varieties of unary semigroups satisfy either $x''\approx x'$ (such as  the domain operation $\DD$ or the range operation $\RR$) or $x''\approx x$ (such as  negation $-$ or conversion ${}^\conv$).  Of this second kind, a broad class is the \emph{involuted semigroups}, where $'$ is usually written as ${}^{-1}$ and the laws $(x^{-1})^{-1}\approx x$ and $(x*y)^{-1}\approx y^{-1}*x^{-1}$ hold.  
Schein~\cite[Proposition~3]{sch} showed that the variety generated by the $\{{}^\conv,\comp\}$-algebras representable as binary relations is precisely the class of involuted semigroups, though the class of  algebras representable as systems of binary relations is a proper quasivariety.  \emph{Inverse semigroups} are a subvariety of involuted semigroups, additionally satisfying $x*x^{-1}*x\approx x$ and $x*x^{-1}*y*y^{-1}\approx y*y^{-1}*x*x^{-1}$.  Inverse semigroups have a classical Cayley style representation known as the Wagner-Preston representation (see~\cite{clipre} for example) that shows they are precisely the algebraic systems isomorphic to algebras of injective partial functions under composition and inverse.  Conveniently, the operations $\DD$ and $\RR$ coincide with $\KL$ and $\KR$ for injective partial functions, 
%indeed $\KLn{n}(a)=\stackrel{n \mbox{ times}}{\overbrace{\KL(a)\comp \ldots\comp \KL(a)}}=\DD(a),\; \KRn{n}(a)=\stackrel{n \mbox{ times}}{\overbrace{\KR(a);\ldots;\KR(a)}}=\RR(a)$  when $a$ is a partial injection, 
so that we are able to revisit a result of Gould and Kambites \cite[Theorem~3.4 and Corollary~4.3]{goukam} that proves the undecidability of representability of bi-unary semigroups as subreducts of inverse semigroups.  
\begin{restatable} {thm}{thmmainbiunary}\label{thm:mainbiunary}
Let $\mathcal{K}$ be any class of unary semigroups in signature $\{*,{}'\}$ that contains the variety of inverse semigroups.
The following problem is undecidable: given a finite bi-unary semigroup ${\bf S}$ in signature $\{\comp,\KL,\KR\}$,  does  ${\bf S}$ embed into the subreduct of a \up[finite\up] semigroup $( T,*,{}' )$ in $\mathcal{K}$, where $x\comp y, \;\KL(x),\; \KR(x)$ are interpreted as $x*y,\;\; x*x', x'*x$, respectively?
\end{restatable}
Theorem~\ref{thm:mainbiunary} captures the Gould and Kambites result as the base case when~$\mathcal{K}$ is the class of inverse semigroups.  But it also captures the $\{\KL,\KR, \comp\}$ case of Theorem~\ref{thm:mainkernel}, when $\mathcal{K}$ is the class of $\{{}^\conv,\comp\}$ algebras of binary relations.  In general though, the result requires no properties at all of the operation $'$ in $\mathcal{K}$, only that $\mathcal{K}$ contain the class of inverse semigroups.

If we restrict to a finite lattice of subsignatures, such as all  $2^9$ subsets of $\{0, 1, -,\allowbreak +, \cdot, \leq, 1', {}^\conv, \comp \}$ ordered by inclusion, each non-empty set of signatures must include a minimal one, by finiteness.  However, if we generalise to allow all signatures whose operations are definable by relation algebra terms, we have an infinite lattice of signatures, ordered by definability.
As our third and final  contribution we apply Theorem \ref{thm:mainbiunary} 
and   provide a signature with UDR and UDFR which is above no  minimal signature with either UDR or UFDR --- we show that there is an infinite sequence of signatures, each term-definable in its predecessor, and each with the UDR and UDFR property, but for which the limit is simply the signature $\{\comp\}$, which has DR and FDR.  

Let $\KLn{n}(a)=\stackrel{n\, \mbox{\tiny times}}{\overbrace{\KL(a)\comp \ldots\comp \KL(a)}},\; \KRn{n}(a)=\stackrel{n\, \mbox{\tiny times}}{\overbrace{\KR(a)\comp \ldots\comp \KR(a)}}$.
The desired sequence can be taken as $\big(\{\comp,\KLn{2^n},\KRn{2^n}\}\big)_{n\in\mathbb{N}}$ in the following theorem.
\begin{restatable}{thm}{thmchain}\label{thm:chain}
The signatures $\{\KLn{n},\KRn{n},\comp\}$ have the UDR and UDFR property, but the only operations on binary relations that are expressible as term functions in $\{\KLn{n},\KRn{n},\comp\}$ for arbitrarily large $n$ are terms expressible in $\{\comp\}$.    The signature~$\set{\comp}$ has DR and DFR.
\end{restatable}

\section{Related work}

The signature $\{\Rightarrow,\comp\}$ with UDR and UDFR was considered by Lewis-Smith and \v{S}emrl~\cite{L-SS} and is minimal if $\Rightarrow$ is allowed as a standard Boolean connective  within the relation algebra signature $\set{0, 1, -, +, \cdots, 1', {}^\conv, \comp }$.   
The \emph{kernel}  $\KL$ and \emph{cokernel} $\KR$ were recently considered by East and Gray~\cite[\S3]{easgra}.  
  The  operation $\join$ was introduced in Bogaerts et al.~\cite{BtCMvdB} where it is noted as equivalent to Lawson's orthogonal join partial operation~\cite{law} on Boolean inverse monoids.

In the context of representability by binary relations, the partial group method was introduced in \cite{hirjac}, but as explained there, it has its roots in a range of semigroup theory applications which in turn are built from the earlier work of the 1950s work of Evans \cite{eva} on the uniform word problem in universal algebra and the 1940s theory developed by Albert~\cite{alb} for Latin Squares and quasigroups.
A flurry of semigroup theoretic developments occurred in the 1990s, after Kublanovsky discovered how to usefully encode Evan's partial group embedding problem into many semigroup-theoretic problems via Albert's homotopy embeddings.  This insight was incorporated into Hall, Kublanovsky, Margolis, Sapir and Trotter~\cite{HKMST}, which concerns the embeddability of semigroups into completely 0-simple semigroups and other related classes, but the work of Sapir~\cite{sap:H} (which predates \cite{HKMST} due to publication delays),  Jackson~\cite{jac:Hembed}, Kublanovsky and Sapir~\cite{kubsap} concerns embeddability in terms of various semigroup-theoretic structural relations, while later Sapir~\cite{sap:amalgam} and Jackson~\cite{jac:amalgam} applied the method to the show undecidability of embeddability of semigroup and ring amalgams.  Each of these applications, as well as subsequent applications to representability as binary relations (such as the present article), uses variants of essentially the same underlying construction from \cite{HKMST}, though finding the right implementation of the right variant can be remarkably subtle and small changes can change the outcome dramatically: this is already clear from the~13 equivalent conditions in the original \cite{HKMST}, and reinforced further by the eventual discovery that condition 13 collapses completely, even if the remaining 12 are indeed equivalent~\cite{jacvol}.

  % Other stuff \nb{Just adding a note here, so we don't forget to delete it}
  
\section{Operations}\label{sec:operations}
Our focus is on signatures weaker than that of relation algebra, however we also consider one constant symbol $\uu$ which is not definable by any relation algebra term.  When $\uu$ is included in the signature, we require that $\uu^\theta=X\times X$, the universal relation, in order for $\theta$ to qualify as a representation over the base $X$.
In this article, our primary definition of $-$ is with respect to this universal relation, in keeping with the idea of representing into the power set $\wp(X\times X)$ for some set $X$.  This ``universal complementation'' should be compared to a localised complementation, where $(-a)^\theta$ is the complement of $a^\theta$ within the union of all other relations that are present; see \cite{neu} for further discussion\footnote{The following problem remains open.
\begin{problem}
Is \up(finite\up) representability  undecidable for complemented semigroups under relativised complement?
\end{problem}
} .  When $1$ is included in the signature $1^\theta$ plays the weaker role of a top (not necessarily universal) binary relation.  
Of course $1$ will coincide with $\uu$ if both are present, because $\uu$ contains all binary relations over $X$, and $1$ is intended to be the top element with respect to~$\leq$.  
But when $\uu$ is not present, then the notation allows $1$ to be interpreted as a top element that may not be the universal relation.  
We mention that there are different philosophical stances on the  semantics of $-$ and $1$.  
In Tarski's original \cite{tar} the element $1$ is intended to be the universal relation, as seems correct in the context of representing within the set of all binary relations on a set $X$.  
In Schein's surveys \cite{sch70} and \cite{sch91}, the element $1$ is typically taken to mean the universal relation, though in \cite{sch91} the symbol  $\uu$ is used.   
However, the top element of a direct product of non-trivial algebras will not be universal. The definition of $1$ is now often taken to mean only the properties that can be forced upon it by the remaining operations.  
In the full Tarski signature this makes $1$ an equivalence relation that contains all other relations, because $1^\conv =1\comp 1=1\geq 1'$ and $x\leq 1$ for all $x$.

In addition, we consider various constants/functions $f$ definable by relation algebra terms, and express the definition by a three variable formula with up to two variables free in a first-order language with binary predicates and equality\footnote{Each three variable formula in this language is equivalent to a relation algebra term \cite[Section 3.9]{TG87} or \cite[Theorem 3.32]{HH:book}}.  
When the RA-term defining the new function $f$ is absent from the signature under consideration, a representation is still required to interpret  $f$ according to the 3 variable definition, thus $\DD(a)^\theta=\set{(x, x)\mid \exists z (x, z)\in a^\theta}$ for any representation $\theta$ where $\DD$ is included in the signature, even if $1'$ or $\cdot$ are absent.

% \begin{equation*}
% \begin{array}{l|lrl}
% f&\mbox{RA-term def.}&\multicolumn{2}{l}{\mbox{3 variable def. in reps.}}\\
% \hline
% \cdot &a\cdot b=-(-a+-b)&(a\cdot b)(x, y)&\iff (a(x, y)\wedge b(x, y))   \\
% 0'&0'=-1'& 0'(x, y)&\iff x\neq y\\
% \Rightarrow& (a\Rightarrow b)= b+-a&(a\Rightarrow b)(x, y)&\iff (b(x, y)\vee \neg a(x, y))\\
% \DD&\DD(a)=1'\cdot a\comp 1&\DD(a)(x, y)&\iff (x=y\wedge\exists z a(x, z))\\
% \RR&\RR(a)=1'\cdot (1\comp a)&\RR(a)(x, y)&\iff (x=y\wedge\exists z a(z, y))\\
% \Ad&\Ad(a)=1'\cdot -\DD(a)&\Ad(a)(x, y)&\iff (x=y\wedge\neg\exists z a(x,z))\\
% \Ar&\Ar(a)=1'\cdot -\RR(a)&\Ar(a)(x, y)&\iff(x=y\wedge\neg\exists z a(z, y))\\
% \KL=\KLn{1}&\KL(a)=a\comp a^\conv&\KL(a)(x, y)&\iff \exists z(a(x, z)\wedge a(y, z))\\
% \KR=\KRn{1}&\KR(a)=a^\conv\comp a&\KR(a)(x, y)&\iff\exists z(a(z, x)\wedge a(z, y))\\
% \KLn{n+1}&\KLn{n+1}(a)=KL(a)\comp \KLn{n}(a)&\KLn{n+1}(a)(x, y)&\iff\exists z(\KL(a)(x, z)\wedge\KLn{n}(z, y))\\
% \KRn{n+1}&\KRn{n+1}(a)=KR(a)\comp \KRn{n}(a)&\KRn{n+1}(a)(x, y)&\iff\exists z(\KR(a)(x, z)\wedge\KRn{n}(z, y))\\
% \join &(a\join b)= a+ \Ad(a)\comp b\comp \Ar(a)&(a\join b)(x, y)&\iff (a(x, y)\vee (b(x, y)\wedge\forall z(\neg a(x, z)\wedge\neg a(z, y)))
% \end{array}
% \end{equation*}

\begin{equation*}
\begin{array}{l|lrl}
f&\mbox{RA-term def.}&\multicolumn{2}{c}{\mbox{3 variable def.\ in reps.}}\\
\hline
\cdot &a\cdot b=-(-a+-b)\hspace{0.8cm}&(a\cdot b)(x, y)&\!\!\!\!\iff (a(x, y)\wedge b(x, y))   \\
0'&0'=-1'& 0'(x, y)&\!\!\!\!\iff x\neq y\\
\Rightarrow& (a\Rightarrow b)= b+-a&(a\Rightarrow b)(x, y)&\!\!\!\!\iff (b(x, y)\vee \neg a(x, y))\\
\DD&\DD(a)=1'\cdot a\comp 1&\DD(a)(x, y)&\!\!\!\!\iff (x=y\wedge\exists z a(x, z))\\
\RR&\RR(a)=1'\cdot (1\comp a)&\RR(a)(x, y)&\!\!\!\!\iff (x=y\wedge\exists z a(z, y))\\
\Ad&\Ad(a)=1'\cdot -\DD(a)&\Ad(a)(x, y)&\!\!\!\!\iff (x=y\wedge\neg\exists z a(x,z))\\
\Ar&\Ar(a)=1'\cdot -\RR(a)&\Ar(a)(x, y)&\!\!\!\!\iff(x=y\wedge\neg\exists z a(z, y))\\
\KL=\KLn{1}&\KL(a)=a\comp a^\conv&\KL(a)(x, y)&\!\!\!\!\iff \exists z(a(x, z)\wedge a(y, z))\\
\KR=\KRn{1}&\KR(a)=a^\conv\comp a&\KR(a)(x, y)&\!\!\!\!\iff\exists z(a(z, x)\wedge a(z, y))\\
\KLn{n+1}&\multicolumn{2}{l}{\KLn{n+1}(a)=\KL(a)\comp\KLn{n}(a)}&\\
&\multicolumn{3}{r}{\KLn{n+1}(a)(x, y)\iff\exists z(\KL(a)(x, z)\wedge\KLn{n}(z, y))}\\
\KRn{n+1}&\multicolumn{2}{l}{\KRn{n+1}(a)=KR(a)\comp \KRn{n}(a)}& \\
& \multicolumn{3}{r}{\KRn{n+1}(a)(x, y) \iff\exists z(\KR(a)(x, z)\wedge\KRn{n}(z, y))}\\
\join & \multicolumn{2}{l}{(a\join b)= a+ \Ad(a)\comp b\comp \Ar(a)}&\\
&\multicolumn{3}{r}{\hspace{0.8cm}(a\join b)(x, y)\iff (a(x, y)\vee (b(x, y)\wedge\forall z(\neg a(x, z)\wedge\neg a(z, y)))}
\end{array}
\end{equation*}

We also consider new relation symbols, definable by relation algebra equations, and give the corresponding 3 variable requirement for representations. 
\[
\begin{array}{l|rll}
\mbox{Relation}&\mbox{Atom}& \mbox{RA term def.}&\mbox{3 var.\ semantic def.}\\
\hline
\leq &(a\leq b)&\!\!\!\!\!\!\iff a+b=b&   \forall x, y (a(x, y)\rightarrow b(x, y))\\
\II&\II(a) &\!\!\!\!\!\!\iff \!\!\!\!\begin{array}{l}\!\!( a\comp a^\conv\leq 1' \\\wedge a^\conv\comp a\leq 1')\end{array}\!\! & \begin{array}{ll}
\!\! \forall x, y, z  & \!\!\!\!\!((a(x, y)\wedge a(x, z))\rightarrow(y=z))\\ &\!\!\!\wedge((a(x, z)\wedge a(y, z))\rightarrow x=y)\end{array}\\
\equiv_{\DD}&(a\equiv_{\DD} b)&\!\!\!\!\!\!\iff \DD(a)=\DD(b)& \forall x (\exists y a(x, y)\leftrightarrow \exists y b(x, y))\\
\equiv_{\RR} &(a\equiv_{\RR} b)&\!\!\!\!\!\!\iff \RR(a)=\RR(b)&\forall y (\exists x a(x, y)\leftrightarrow\exists x b(x, y))
\end{array}
\]

When we consider signatures without converse, identity and order,  the term definition $a\comp a^\conv\leq 1'\wedge a^\conv\comp a\leq 1'$ of $\II(a)$ is not available.  
The central trick in~\cite{hirjac} was  $\theta_2(a)$ below for signatures which may omit converse and  include diversity,  meet and composition.  
When $\leq$ and $-$ are present, then the property $(a\comp 0')\cdot a=0$ can be expressed as $-(a\comp 0')\geq a$, and similarly for $(0'\comp a)\cdot a$ as $-(0'\comp a)\geq a$, which is used by Neuzerling~\cite{neu}, and further embellished in $\theta_3(a)$ avoiding reference to any order-theoretic relations or operations,

\begin{lem}\label{lem:inj}
Consider the following formulas and their signatures \up($n\geq 1$\up):
\[
\begin{array}{l ll}
\II_i(a)&\mbox{definition}&\Sigma_i\\
\hline
\II_0(a)& (a\comp a^\conv\leq 1'\wedge a^\conv\comp a\leq 1')&\set{1',{}^\conv,\comp , \leq}\\
\label{inj1}\II_1(a)& (\KLn{n}(a)\leq 1'\wedge\KRn{n}(a)\leq 1')&\set{1',\KLn{n}(a),\KRn{n}(a), \leq}\\
\II_2(a)&((a\comp 0')\cdot a=0\wedge (0'\comp a)\cdot a=0)&\set{0,0',\cdot,\comp }\\
\II_3(a)&((-(a\comp 0'))\comp\uu\comp(-(0'\comp a))=\uu)&\set{0',\uu,-,\comp }
\end{array}
\]
For $i=0,1,2,3$, if $\c A$ is a $\Sigma_i$-structure, $a\in\c A$ and $\theta$ is a $\Sigma_i$-representation of $\c A$ over $X$ then $a^\theta$ is an injection iff $\c A\models \II_i(a)$.
\end{lem}
\begin{proof}
Suppose $a^\theta$ is an injection.    In $\Rel(X)$ we have
\[ a^\theta\comp  (a^\theta)^{-1}=\DD(a^\theta)=\KLn{n}(a^\theta),\; (a^\theta)^{-1}\comp a^\theta=\RR(a^\theta)=\KRn{n}(a^\theta).\]
For $i=0,1$, since $\DD(a^\theta), \RR(a^\theta), \KLn{n}(a^\theta), \KRn{n}(a^\theta) \subseteq Id_X=\set{(x, x)\mid x\in X}$, it follows that  $\c A\models\II_0(a)\wedge\II_1(a)$. For $i=2$, since $a^\theta$ is a function we have $a\comp 0'\;\cdot\; a =0$ and since it is an injection we have $0'\comp a\;\cdot\; a=0$, so $\c A\models\II_2(a)$.  Finally for $i=3$, as $(0')^\theta$ is the $\neq$ relation, and $a^\theta$ is a  partial function then $-(a\comp 0')^\theta$ agrees with $a^\theta$ on the domain of $a^\theta$ and relates all points outside of the domain of $a^\theta$ to all of $X$.  
In particular, $-(a\comp 0')^\theta$ has domain $X$ so that $(-(a\comp 0')\comp \uu)^\theta =\uu^\theta$.  Likewise, since $a$ is injective,  $(\uu\comp-(0'\comp a))^\theta =\uu^\theta=X\times X$,  hence $\c A\models\II_3(a)$.

Conversely,  now assume that $a^\theta$ is not a function: so there are $x, y,z$ with $y\neq z$,\/ $(x,y)\in a^\theta$ and $(x,z)\in a^\theta$.    Then $(y, z)\in (a^\theta)^{-1}\comp  a^\theta\setminus (1')^\theta$ so $\c A \not\models\II_0(a)$, \/ $(y, z)\in\KRn{n}(a^\theta)\setminus (1')^\theta$ so $\c A\not\models\II_1(a)$.  Also, $(x, y)\in a^\theta\cap (a\comp 0')^\theta$ so $\c A\not\models\II_2(a)$.  For $i=3$,  for all $w$  either $w\neq y$ or $w\neq z$ so we have $(x,w)\in (a\comp 0')^\theta$, so that $x$ is not in the domain of $-(a\comp 0')^\theta$.  
Then $((-(a\comp 0'))\comp \uu)^\theta$ has strictly smaller domain than $\uu^\theta$, as does $(-(a\comp 0'))\comp \uu\comp (-(0'\comp a))^\theta$, so $\c A\not\models\II_3(a)$.  
Dually, if $a^\theta$ is not injective then  $\II_i(a)$ fails. 
\end{proof}

\section{The partial group embedding method}
An instance $(P, *)$ of the \emph{partial group [finite] embedding problem} is a partial binary function $*$ over a finite set $P$, it is a yes instance of $(P, *)$ extends to a [finite] group else it is a no-instance.  Both the partial group embedding problem  and the partial group finite embedding problem are known to be undecidable \cite{eva, slob}.   In the following, we restrict instances of these problems to \emph{square partial groups} $(P, *)$, where 
\begin{itemize}
\item there are no explicit violations of associativity,
\item there is an element $e\in P$ such that $e*a=a*e=a$ for all $a\in P$,
\item there is a subset $\sqrt{P}\subseteq P$ containing $e$ such that $*$ is a total function from $\sqrt{P}^2$ to $P$ and is undefined for products involving $P\backslash \sqrt{P}$ except for those of the form $e*a=a*e=a$,
\item for each $a\in\sqrt{P}$ there is a unique $a'$ such that $a*a'=a'*a=e$.  
\end{itemize}
and both problems remain undecidable \cite[Lemma~2.4]{hirjac}. Notationally we set $\{e\}=\sqrt{P}^{(0)}$, $\sqrt{P}^{(1)}:=\sqrt{P}$ and $\sqrt{P}^{(2)}:=P$. All the undecidability results we establish here are based on reductions from  the square partial group [finite] embedding problem.

A detailed motivation behind the underlying construction, in the context of representability is given in \cite{HHJ}.  
In the present article we take a slightly different approach to previous efforts in representability applications, and use the particular formulation in Sapir~\cite{sap:H}.
We recall some relevant semigroup-theoretic notions. 

\medskip
Let $S=(S, \comp)$ be any semigroup and define the relation $\leq_\mathscr{L}$ by $a\leq_\mathscr{L}b$ if $a=b$ or there exist $x$ such that $x\comp b=a$.  In other words $a\leq_\mathscr{L}b$ if $a$ lies in the principal left ideal of $b$.  
Dually we define $\leq_\mathscr{R}$ in terms of right multiplication.  Associativity easily shows that these are preorders: reflexive and transitive.  
The binary relations~$\mathscr{L}$ and~$\mathscr{R}$ are then defined as the equivalence relations determined by $\leq_\mathscr{L}$ and $\leq_\mathscr{R}$ respectively: $a\mathrel{\mathscr{L}}b\iff (a\leq_\mathscr{L}b \wedge b\leq_\mathscr{L}a)$  and $a_\mathscr{R} b \iff(a\leq_\mathscr{R} b\wedge b\leq_\mathscr{R} a)$.
The relation $\mathscr{H}$ is defined by $\mathscr{L}\cap \mathscr{R}$.  
These and some other relations are usually referred to as \emph{Green's relations}, see any semigroup theory text, such as \cite{clipre}.  
The article \cite{sap:H} concerns the problem of when a binary relation $R\subseteq S\times S$ on a finite semigroup ${\bf S}$ might sit within the $\mathscr{H}$-relation in some larger embedding semigroup ${\bf T}\geq {\bf S}$.  This problem turns out to be undecidable, even in the case where $R$ consists of a total relation on a subset $A\subseteq S$~\cite{jac:Hembed}.

A \emph{homotopy} from a square partial group $(P, *)$ into a semigroup ${\bf S}$ is a triple of maps $(\alpha,\beta,\gamma)$ where $\alpha, \beta:\sqrt{P}\rightarrow{\bf S}, \;\gamma:P\rightarrow {\bf S}$ such that  $\alpha(a)\comp \beta(b)=\gamma(a\comp b)$,  and is called a \emph{homotopy embedding} if $\gamma$ is injective.  
An \emph{$\mathscr{H}$-embedding} of $(P, *)$ into a semigroup ${\bf S}$ is a homotopy embedding in which each of $\alpha(\sqrt{P})$, $\beta(\sqrt{P})$, $\gamma(P)$ lies within an equivalence class of Green's $\mathscr{H}$-relation on ${\bf S}$.  
The argument in \cite[Proposition 5.1]{hirjac} is loosely equivalent to the proof of the following lemma due to Mark Sapir, but we have found this particular formulation in terms of $\mathscr{H}$-embedding provides the most streamlined  presentation of the results here.
\begin{lem} \label{lem:sap} \up(M.~Sapir \cite{sap:H}\up)
Let $(P, *)$ be a square partial group.  There is an $\mathscr{H}$-embedding of $(P, *)$ into a \up[finite\up] semigroup if and only if $(P, *)$ embeds into a \up[finite\up] group.
\end{lem}
%The simplest construction for creating a semigroup from a square partial group ${\bf P}=\langle P,\comp)$ is 

The next lemma reveals why we are interested in the combination $\{\II,\equiv_\DD,\equiv_\RR, \comp\}$: it enables application of Lemma \ref{lem:sap}.
\begin{lem}\label{lem:H}
Let $S$ be a signature that can express $\comp, \II$ and $\equiv_\DD$, let $\Rel(X)$ be the $S$-algebra of all binary relations over $X$.  
For all binary relations $r, s\in \Rel(X)$, if   $\II(r), \; \II(s)$ and $ r\equiv_\DD s$   then $\Rel(X) \models r \mathrel{\mathscr{R}}s$. Dually, if $S$ can express $\comp, \II$ and $\equiv_{\RR}$ and $r\equiv_\RR s$ then $\Rel(X)\models r\mathrel{\mathscr{L}}s$.
\end{lem}
\begin{proof}
Since $\II(r)$ holds, $r$ is an injective function, with an inverse $(r)^\conv\in\Rel(X)$, and  $r\comp (r)^\conv$ is the identity on the domain of $r$.  If additionally $r\equiv_\DD s$ then domain of~$r$ coincides with the domain of $s$, so that $r\comp ((r^\conv)\comp s)=s$, hence  $r\leq_{\mathscr{R}} s$.  Similarly, $s\leq_{\mathscr{R}} r$, so $r{\mathscr{R}}s$.  Dually,  if  $r, s$ are injections and $r\equiv_\RR s$ then $r\mathrel{\mathscr{L}}s$.
\end{proof}

\section{Constructions}\label{sec:constructions}
Let $\P=(P, *)$ be a square partial group.  [The following constructions are  obtained as subreducts of the Boolean monoid $M(\P)$ of \cite{hirjac}, expanded to include kernel and cokernel in the signature.]    Consider the semigroup $( E(\P), \comp )$ where 
\begin{align*}
     E(\P)&=\set 0\cup \set{a_{i,j}\mid i\leq j<3,\; a\in\sqrt P^{j-i}}\\
    &=\set0\cup \set{e_{i,i}\mid i<3}\cup \set{a_{0,1}, b_{1,2}, c_{0, 2}\mid a, b\in\sqrt{P},\; c\in P}
\end{align*}
and $\comp$ is defined by
\begin{equation*}
a_{i, j}\comp b_{j, k} = (a*b)_{i,k}
\end{equation*}
for $i\leq j\leq k<3,\; a\in\sqrt P^{j-i}, \; b\in\sqrt P^{k-j}$, and all other compositions yield $0$.  We may also define $\cdot, \DD, \RR, \KL,\KR$ by
\begin{align*}
x\cdot y&=\left\{\begin{array}{ll}x&x=y\\0&x\neq y\end{array}\right.\\
    \DD(a_{i, j})=\KL(a_{i, j})&=e_{i, i}&\RR(a_{i, j})=\KR(a_{i, j})=e_{j, j}
\end{align*}
for $a\in\sqrt P^{j-i}$, and $\DD(0)=\RR(0)=\KL(0)=\KR(0)=0$. Let
\begin{equation}\label{eq:m0}
    \E_0(\P)=( E(\P), 0, \cdot, \DD, \RR, \KL, \KR, \comp).
\end{equation}

\begin{remark}
 The structure $\E_0(\P)$ can be used to prove undecidability for signatures $\set{\KL, \KR, \comp} \subseteq \tau \subseteq\set{0, \cdot, \DD, \RR, \KL, \KR, \comp}$.  
 \end{remark}
 \begin{remark}
 The triple of maps $(\alpha, \beta,\gamma)$ where
\[ \alpha(a)=a_{0,1},\; \beta(b)=b_{1,2},\;\gamma(c)=c_{0,2}\]
for $a, b\in\sqrt{P},\; c\in P$,  is a homotopy embedding  from $\P$  to $(  E(\P), \comp )$.  However, except in the trivial case, it will not be an $\mathscr{H}$-embedding --- for any distinct $a\neq b\in {\bf P}$ we have $a_{0, 1}\not\!\!\!\mathscr{L} b_{0, 1},\;a_{0, 1}\not\!\!\!\mathscr{R} b_{0, 1} $ as there are no solutions of $x\comp a_{0, 1}=b_{0, 1}$ or $a_{0,1}\comp y=b_{0,1}$ in $ E({\bf P})$, so $\alpha(\sqrt{P})$ is not contained in any $\mathscr{H}$-equivalence class.  In order to make $(\alpha, \beta, \gamma)$ into an $\mathscr{H}$-embedding, the codomain $ E({\bf P})$ must be properly extended to a semigroup which includes the required solutions $x, y$, among other things.
\end{remark}

Next, we expand the base and signature of $ \E_0(\P)$ to include negation.
For each $i\leq j<3,\; a\in\sqrt P^{j-i}$ let $\bar a_{i, j}, 1_{i, j}, 1_{j, i}$  be new elements.  Let
\[  E_1(\P)=\set{ a_{i, j}, \bar a_{i,j}\mid i\leq j<3,\; a\in \sqrt P^{j-i}}\cup\set{1_{i, j}\mid i, j<3}\]
and extend the definition of $\comp$ by

\begin{align*}
    \bar a_{i, j}\comp x_{j, k}&=1_{i, k}&i\leq j<3,\; k<3,\; a\in\sqrt P^{j-i}\\
    1_{i, j}\comp x_{j, k}=x_{i,j}\comp 1_{j,k} &=1_{i, k}& i, j, k<3,\; x_{i, j}\neq 0\\
    x_{i, j}\comp \bar a_{j,k}&= 1_{i,k}&j\leq k<3,\; i<3
\end{align*}
and let all remaining compositions over $ E_1(\P)$ yield $0$.
Define an ordering $\leq$ over $ E_1(\P)$ by letting $0\leq x$ (all $x\in E_1(\P)$),  $a_{i, j}\leq 1_{i, j}, \;  \bar a_{i, j}\leq 1_{i, j}, $ for $i\leq j<3$.  For $i, j<3$ let  $-_{(ij)}$ swap $0$ and $1_{i, j}$, and for $i\leq j<3$ it also swaps  $a_{i,j}$ and $\bar a_{i, j}$.    Extend $\DD,\RR,\KL,\KR$ to the new elements, by 
\begin{align*}
    \DD(\bar a_{i, j})&=e_{i, i}&\RR(\bar a_{i, j})&=e_{j,j}\\
    \KL(\bar a_{i, j})&=1_{i, i}&\KR(\bar a_{i, j})&=1_{j, j}
\end{align*}

Let $ E_1^\Sigma(\P)$ consist of all formal sums $\Sigma S$ where $S\subseteq  E_1(\P)$, for each $i<3$ there is at most one $j<3$ such that $a_{i, j}$  is in $S$ (for any $a\in\sqrt P^{j-i}$) and for each $j<3$ there is at most one $i<3$ such that $a_{i, j}$  is in $S$.  Note, we are not including $+, \Sigma$ in the signature, just using the symbols to name elements.  Functions $\DD, \RR, \KL, \KR, \comp$ may be extended to $ E_1^\Sigma(\P)$ by an additive convention, so for example $\DD(\bar a_{0, 1}+e_{2,2})\comp ((\KL\bar e_{1,1})+\bar e_{2,2})=(e_{0, 0}+e_{2,2})\comp(1_{1,1}+e_{2,2})=e_{2,2}$.

Let 
\begin{align*}
1&=\Sigma_{i, j<3} 1_{i, j}\\
    1'&=e_{0,0}+e_{1,1}+e_{2,2}\\
    \Ad(\Sigma S)&=\Sigma\set{e_{i, i}\mid  i<3,\; e_{i, }\comp \DD(\Sigma S)=0}\\
        \Ar(\Sigma S)&=\Sigma\set{e_{j, j}\mid  j<3,\;\DD(\Sigma S)\comp e_{j, j}=0}\\
        \Sigma S\join\Sigma T & =\Sigma(S\cup\set{t\in T\mid \DD(t)\comp\DD(\Sigma S)=\RR(t)\comp\RR(\Sigma S)=0}.\\
\end{align*}
We may write a formal sum $\Sigma S$ as $\Sigma_{i, j<3} x_{i, j}$,  but the restriction of  the uniqueness of any $x_{i, j}=a_{i, j}$ for given $i$ or for given $j$ still applies.  
Define an order over these sums by
\[\Sigma_{i,  j<3}x_{i, j}\leq \Sigma_{i,  j<3}y_{i, j} \iff \forall i, j<3(x_{i,j}\leq y_{i, j})\]
and define negation over $E_1^\Sigma(\P)$ by
\[ -\Sigma_{i, j<3}x_{i,j}=\Sigma_{i,  j<3} -_{(ij)} x_{i,j}\]

  Let 
  \begin{equation}\label{eq:m1}
       \E_1(\P)= (E_1^\Sigma(P), 0, 1, \leq, -,0',1', \DD, \RR, \Ad, \Ar, \KL, \KR, \join, \comp).
  \end{equation}
  For signatures including intersection and diversity we slightly modify this expansion.   For $i\leq j<3$ and $A\subseteq\sqrt P^{j-i}$ let $\overline A_{i, j}$ be a new element (thought of as the sum of all $g_{i, j}$ for $g\not\in A$ in some extension group). For $i\leq j<3$ we write $1_{i, j}$ for $\overline{\emptyset}_{i, j}$.   Additionally we include elements $1_{j, i}$ for $i<j<3$.  Let $E_2(\P)=E(\P)\cup \overline E(\P)\cup \set{1_{j, i}\mid i<j<3}$, where 
  \[\overline E(\P)=\set{\overline{A}_{i, j}\mid  i\leq j<3,\; A\subseteq \sqrt P^{j-i}}.\]  Consider the structure $(E_2(\P), 0,\cdot,   \DD, \RR, \KL,\KR, \comp)$, where meet is defined by $x\cdot x=x$ (all $x\in E_2(\P)$), \/ $\overline A_{i, j}\cdot \overline B_{i, j}= \overline{(A\cup B)}_{i, j}$ and in all other cases $x\cdot y=0$.  Other functions are defined by letting $0$ be a zero for all functions, and 
\begin{align*}
   \DD(a_{i,j})=\DD(\overline A_{i, j})&=e_{i, i}&
   \RR(a_{i, j})= \RR(\overline {A}_{i, j})& = e_{j, j}\\
  \KL(a_{i, j})&=e_{i, i}&
  \KL(\overline A_{i, j})=\KL(1_{i, j})&= 1_{i, i}\\
  \KR(a_{i, j})&=e_{j, j}&
  \KR(\overline A_{i, j})=\KR(1_{i, j})&= 1_{j, j}\\
  \overline{A}_{i, j}\comp b_{j, k}&=\overline{\set{a\comp b\mid a\in A}}_{i, k}&\overline{A}_{i, j}
  \comp\overline{B}_{j, k}&= 1_{i, k}
\end{align*}
and all remaining compositions yield $0$.

We  expand the base further to allow $\Ad, \Ar, \join, 0', 1'$ to be defined, as before.     
Let $ E_2^\Sigma(\P)$ consist of all formal sums $\Sigma_{i,  j<3} x_{i, j}$ where $x_{i, j}\in E_2(\P)$, for each $i<3$ there is at most one $j$ with $i\leq j<3$ and $x_{i, j}\in E(\P)\cup \overline E(\P)$, and for each $j<3$ there is at most one $i\leq j$ where $x_{i, j}\in E(\P)\cup \overline E(\P)$.    

Meet, domain,  range,  composition and kernels extend to $  E_2^\Sigma(\P)$ by an additive convention: $(\Sigma_{i\leq j<3} x_{i, j})\cdot(\Sigma_{i\leq j<3} y_{i, j})=\Sigma_{i\leq j<3}(x_{i,j}\cdot y_{i, j})$, \/ $\KL(\Sigma_{i\leq j<3}x_{i,j})=\Sigma_{i\leq j<3}\KL(x_{i, j}), \; \DD(\Sigma_{i\leq j<3} x_{i, j})=\Sigma_{i\leq j<3}\DD(x_{i, j}), \; (\Sigma_{i\leq j<3} x_{i, j})\comp (\Sigma_{i\leq j<3}y_{i, j})=\Sigma_{i\leq j\leq k<3}(x_{i,j}\comp y_{j, k})$ etc.   Let
 \begin{align*}
  1'&=e_{0,0}+e_{1,1}+e_{2,2}\\
  1&=\Sigma_{i\leq j<3} 1_{i, j}\\
  0'&=\Sigma_{i<3}\overline {e}_{i, i}+\Sigma_{i\neq j<3} 1_{i, j}.
  \end{align*}
Elements below $1'$  have complements in $E_2^\Sigma(\P)$, 
  $ \Sigma_{i\in I}e_{i,i} - \Sigma_{j\in J} e_{j,j} = \Sigma_{k\in I\setminus J} e_{k,k}$
  which allows us to extend  antidomain, antirange and $\join$ by
  \begin{align*} 
    \Ad(a)&=1'-\DD(a)\\
    \Ar(a)&=1'-\RR(a)\\
    \Sigma S\join \Sigma T&= 
    \Sigma(S\cup\set{t\in T\mid \DD (t)\comp\DD(\Sigma S)=\RR (t)\comp\RR(\Sigma S)=0}) 
 \end{align*} 
 we get a structure 
 \begin{equation}\label{eq:m2}
      \E_2(\P)=(  E_2^\Sigma(\P), 0,1, \cdot, 1',0',  \DD, \RR, \KL, \KR,  \Ad, \Ar, \join, \comp ).
 \end{equation}

\section{Proof of main Theorems}
\begin{defn}\label{def:bm}
Let $(G, *)$ be a  group where $|G|\geq 3$.   Let  $B_3(\G)$ be the inverse semigroup $(\set0\cup\set{g_{i,j}\mid i, j<3}, *, {}^{-1})$ where $*$ is defined  by $g_{i, j}*h_{j,k}=(g*h)_{i,k}$ for $i, j<3,\; g, h\in G$, all other products yield $0$, and $(g_{i, j})^{-1}=(g^{-1})_{j, i}$.

    Define $\theta:B_3(\G)\rightarrow\wp((3\times G)\times(3\times G))$ by
$0^\theta=\emptyset$ and
\begin{equation}\label{eq:theta}    g_{i, j}^\theta=\set{((i, h), (j, (h*g)))\mid h\in G},\end{equation} 
for $i\leq j<3$ and $g\in G$. 

\end{defn}
 Observe that    $g_{i, j}^\theta$ is a bijection from $i\times G$ to $j\times G$, and $(a*b)^\theta=a^\theta\comp b^\theta$, for $a, b\in B_3(\G)$, in fact $\theta$ is a generalised  Caley representation. 
  
  For all (partial) injections $\iota$  over $X$ we have 
\begin{equation}
    \label{eq:domker}
 (x, y)\in\DD(\iota)\iff (x=y\wedge\exists z (x, y)\in\iota)\iff (x, y)\in \KL(\iota)\end{equation}
where $\DD, \KL$ are evaluated in $\Rel(X)$, and a dual condition for co-kernel and range of $\iota$.

Also, for $g_{i, j}\in B_3(\G)$, \begin{equation}\label{eq:negker}
\KL(((i\times G)\times(j\times G))\setminus g_{i, j}^\theta)=(i\times G)\times(i\times G)
\end{equation}
using $|G|\geq 3$.

As an example of the power of Lemma \ref{lem:H}, we now prove Theorem \ref{thm:mainbiunary} using $\KL$ and $\KR$, which we recall  for the reader's convenience.

\thmmainbiunary*

\begin{proof} We show that the partial [finite] group embedding problem reduces to the problem in the theorem.    The reduction maps a finite square partial group $\P=(P, *)$ to the $\set{\KL, \KR, \comp}$-structure   $ \E_0(\P)=(E(\P), \KL, \KR, \comp)$ (formally, a reduct of what we defined in \eqref{eq:m0}).
Suppose $\P$ embeds into a [finite] group $(G,*)$. Then $E(\P)\subseteq B_3(G)$ (see Definition~\ref{def:bm}) and  $x\comp y=x*y,\; \KL(x)=\DD(x)=x*x^{-1},\;\KR(x)=\RR(x)=x^{-1}*x$                                                                                                                                                              by \eqref{eq:domker} and its dual.  Since $B_3(G)$ is a [finite] inverse semigroup, $ \E_0(\P)$ is a yes-instance of the Theorem's [finite] embedding problem. 

 Conversely,  suppose that  $ \E_0(\P)$ is a  $\{\KL,\KR,\comp\}$ subreduct of a \up[finite\up] unary semigroup $\lrangle{T, *, {}'}$ where  $x\comp y=x*y,\; \KL(x)=x*x'$ and $\KR(x)=x'*x$.  
Note that these term functions immediately give $\KL(x)\leq_{\mathscr{R}}x$ and $\KR(x)\leq_{\mathscr{L}}x$ for all $x\in T$.  
In $ E({\bf P})$ we have $\KL(a_{ij})=e_{ii}$, and $\KR(a_{ij})=e_{jj}$ by definition.  As $e_{ii}\comp a_{ij}=(e*a)_{ij}=a_{ij}$ and $a_{ij}\comp e_{jj}=(a*e)_{ij}=a_{ij}$ we have $a_{ij}\leq_{\mathscr{R}}e_{ii}=\KL(a_{ij})$ and $a_{ij}\leq_{\mathscr{L}}e_{jj}=\KR(a_{ij})$.  
So $a_{ij}\mathrel{\mathscr{L}}e_{ii}$ and $a_{ij}\mathrel{\mathscr{R}}e_{jj}$ in ${T}$.  
As $a$ was arbitrary and $\mathscr{L}$ and $\mathscr{R}$ are equivalence relations, we have that for any fixed $0\leq i\leq j<3$, all elements of the set $\{a_{ij}\mid a\in P\}$ are 
simultaneously $\mathscr{L}$-related and $\mathscr{R}$-related within~${\bf T}$.  
In other words, $\{a_{ij}\mid  a\in P\}$ lies within an $\mathscr{H}$-equivalence-class of~${ T}$ for any fixed $i,j$, 
so the homotopy  $(a\mapsto a_{0, 1}, \; b\mapsto b_{1,2},\; c\mapsto c_{0,2})$ (where $a, b\in \sqrt P,\; c\in P$) is an $\mathscr{H}$-embedding of ${\bf P}$ into $ E({\bf P})$.  
By Lemma~\ref{lem:sap}, ${\bf P}$ extends to a [finite] group, as required.
\end{proof}

We turn to the proof of  Theorem~\ref{thm:maincomplement}. We  begin by showing that the universal relation and empty relation are abstractly definable in the signature $\{\comp,-\}$.  
By a \emph{local equivalence relation} on a set~$X$, we mean a  symmetric transitive relation, i.e. a relation that is an equivalence relation on some subset $Y$ of $X$ and contains no tuples $(x,y)$ for which $x$ or $y$ are in $X\backslash Y$. 
\begin{lem}\label{lem:1}
Let $\theta$ be a $\{-,\comp\}$ representation of an algebra ${\bf A}=( A,-,\comp)$ as binary relations on a set $X$.
If $e\in A$ has $e\comp e=e$ and $(-e)\comp e=-e=e\comp (-e)$ then
\begin{enumerate}
\item  $e^\theta$ is a nonempty local equivalence relation, and
\item if $(-e)\comp (-e)=-e$ as well, then $e^\theta$ is the universal relation and $-e^\theta$ is the empty relation.
\end{enumerate}
\end{lem}
\begin{proof}
(1) We first show that $e^\theta$ is symmetric.  Assume that $(x,y)\in e^\theta$.  If $(y,x)\in -e^\theta$ then $(y,y)\in (-e\comp e)^\theta=-e^\theta$ so that $(x,y)\in (e\comp -e)^\theta=-e^\theta$, contradicting $(x,y)\in e^\theta$ and the definition of $-$.  
Thus $e^\theta$ is symmetric and from $e\comp e=e$ it is also transitive, hence a local equivalence relation.   Nonemptiness follows from $(-e)\comp e=-e=e\comp (-e)$.

(2)  Now assume the additional property $(-e)\comp (-e)=-e$. Let $x$ be any point in the (nonempty) domain of $e^\theta$.  
Note that as $e^\theta$ is symmetric, so also is $(-e)^\theta$.  
For all $y\in X$, if $(x, y)\in (-e)^\theta$ then by symmetry  of $(-e)^\theta$ and $(-e)\comp(-e)=-e$ we deduce that $(x, x)\in (-e)^\theta$, contradicting our assumption that $x$ is in the domain of the local equivalence relation $e^\theta$.  
Thus, for all $y, z\in X$ we have $(x, y), (x, z)\in e^\theta$.  Again using the symmetry and transitivity of $e^\theta$ it follows that $(y, z)\in e^\theta$, for arbitrary $y, z\in X$.  Thus $e^\theta=\uu$ and $(-e)^\theta = \emptyset$.
\end{proof}
By a Boolean monoid we mean an algebra ${\bf M}=( M,0, 1, -, +,\cdot, 1', \comp)$ where $( M,-, +,\cdot, 0,1)$ is a Boolean algebra,  $\comp$ is additive, $0$ is a left and right zero and~$1'$ is an identity for composition.  
If in addition $1\comp x\comp 1=1$ for all $x\neq 0$ then the Boolean monoid is \emph{simple}. 
%We let $\DD(x)$ denote the term function $(x\comp 1)\cdot 1'$ and $\RR(x)$ the term function $(1\comp x)\cdot 1'$.  
A representation of a Boolean monoid is usually defined to  be an isomorphism from the Boolean monoid to a Boolean monoid of 
binary relations over some set $X$ with natural operators, where $1$ is represented as a top equivalence relation and $-$ is represented as complementation
relative to this top.  However, here we insist  that $-$ is interpreted as complementation in $X\times X$, so that $1$ has to be represented as  $X\times X$.  Only simple Boolean monoids can have representations $\theta$ where
$1^\theta=X\times X$, and every representable, simple Boolean monoid has such a representation.
Note that the element $1$ in a Boolean monoid satisfies the conditions (1) and (2) of Lemma \ref{lem:1}, so that the lemma implies that any $\{-, \comp\}$-representation (where $-$ is represented a complement in the universal relation) is a reduct of a $\{-,\uu, \comp\}$-representation.
\begin{lem}\label{lem:2}
Let ${\bf M} = (M,0,1,-,+,\cdot, 1', \comp)$ be a normal, simple Boolean monoid and $\theta$ a $\{-,\comp\}$-representation for ${\bf M}$ over $X$.  Then there is a $\set{-,\comp}$-representation $\vartheta$ of ${\bf M}$ over a quotient of $X$ in which $1'^\vartheta$ is the identity relation.  
\end{lem}
\begin{proof}
As already noted, Lemma \ref{lem:1} shows that  $1^\theta$ is the universal relation.  
Then $1'^\theta$ has domain and range equal to $X$ because $1'\comp 1=1=1\comp 1'$ and the domain and range of $1^\theta$ is~$X$.  
Now $1'\comp 1'=1'$ and $-1'\comp 1'=1'\comp -1'=-1'$ so that $1'^\theta$ is a local equivalence relation by Lemma \ref{lem:1}, hence a total equivalence relation since $(1')^\theta$ is total.  
Now we may represent over the quotient set $X/(1'^\theta)$, whose elements are the equivalence classes of $1'^\theta$; let $\vartheta$ denote the mapping ${\bf M}\into\wp(X/(1'^\theta))$ given by $(x/1'^\theta,y/1'^\theta)\in m^\vartheta$ if $(x,y)\in m^\theta$, for $m\in M$.   
Because $1'\comp m=m\comp 1'=m$ for all $m\in M$ it is easily seen that $\vartheta$ is a well-defined, injective  
$\{-,\comp\}$ homomorphism from $( M,-,\comp)$ into 
$({\wp(X/(1'^\theta)\times X/(1'^\theta)), \setminus,\comp\;})$.
\end{proof}
Lemma \ref{lem:2} shows that there is no loss of generality in assuming that a $\{-,\comp\}$-representation of a normal Boolean monoid is a $\{-,1',\comp\}$-representation.  
Note that if $X$ is finite then so also is any quotient of $X$.  
Recall the definition $0'=-1'$.

\begin{lem}\label{lem:4}
Let $\theta$ be a $\{-,1', \comp\}$-representation of a normal Boolean monoid ${\bf M}$.
%and let $\texttt{d}$ denote\nb{$0'/d$?} $-1'$.  
If $e\comp 1=a \comp 1$ and $e$ is an idempotent element such that $\II_3(e)$ holds, then $e^\theta=\DD(a^\theta)=\set{(x, x)\mid x\in\dom(a^\theta)}$.  
If $1\comp e=1\comp a$  and $e$ is an idempotent  element where $\II_3(a)$ holds, then $e^\theta=\set{(y, y)\mid y\in\ran (a^\theta)}$.  
\end{lem}

\begin{proof}
    Assume the conditions, so $e^\theta$ is injective, idempotent $e\comp e=e$.  Suppose $(x, y)\in e^\theta=e^\theta\comp e^\theta$.  Then there is $z$ such that $(x, z), (z, y) \in e^\theta$.  Since $e^\theta$ is an injection, $z=y$ and $z=x$, so $x=y$.  This proves that all idempotent injections are restrictions of the identity.  
    From $e\comp 1= a\comp 1$ we deduce $e^\theta=\DD(a^\theta)$, the final sentence is proved dually.
\end{proof}

In view of Lemma \ref{lem:inj} and \ref{lem:4}, the following is a restatement of Lemma \ref{lem:H}.
\begin{lem}\label{lem:5}
Let ${\bf M}$ be a simple Boolean monoid, representable in the signature $\{-,1', \comp\}$.   If $a,b$ are injective elements such that $\DD(a)=\DD(b)$ then $a\mathrel{\mathscr{R}}b$.  Similarly, if $a, b$ are injective and $\RR(a)=\RR(b)$ then $a\mathrel{\mathscr{L}}b$.
\end{lem}

\begin{lem}\label{lem:zz}
Let ${\bf P}$ be a square partial group.  If $ \E_1({\bf P})$ is $\set{-, \comp}$-representable \up[over a finite base\up] then  ${\bf P}$ embeds into a \up[finite\up] group.   
\end{lem}
\begin{proof}
If $ \E_1({\bf P})$ is $\set{\comp , -}$-representable, by Lemma \ref{lem:5} all injective elements with common domain and common range are $\mathscr H$-equivalent, hence $(a\mapsto a_{0,1}, b\mapsto b_{1,2}, c\mapsto c_{0,2})$ is an $\mathscr H$-embedding of $\P$ into $ E(\P)$.  By Lemma \ref{lem:sap}, ${\bf P}$ embeds into a [finite] group.    
\end{proof}
%\begin{thm}\label{thm:complement}
%Membership of the class of finite $\set{;, -}$-representable structures is undecidable.\nb{Already proved?}
%The following are equivalent, and as the first item is undecidable, so too are each of the remaining items.
%\begin{enumerate}
%\item the square partial group ${\bf P}$ embeds into a \up[finite\up] group,
%\item $M({\bf P})$ embeds into the $\{\comp,-\}$ reduct of a \up[finite\up] simple relation algebra,
%\item $M({\bf P})$ has a $\{\comp,-\}$-representation \up[over a finite base set\up],
%\item $M({\bf P})$ embeds into the Boolean monoid reduct of a \up[finite\up] simple relation algebra,
%\item $M({\bf P})$ has a Boolean monoid representation \up[over a finite base set\up].
%\end{enumerate}  
%
%
%In particular, representability and finite representability are undecidable for any subsignature of the Boolean monoid signature containing $\{\comp,-\}$.
%\end{thm}

%We now recall Theorems~\ref{thm:maincomplement},~\ref{thm:mainkernel}, restated below for the reader's convenience, and provide proofs for both theorems. 
 Recall the first main theorem.
\thmmaincomplement*

\begin{proof}
Suppose $\P$ embeds into a [finite] group ${\bf G}= (G, *)$, without loss suppose $P\subseteq G$.     We may replace ${\bf G}$ by ${\bf Z}_3\times{\bf G}$ if necessary, so that 
\begin{equation}
   \label{eq:G} (G\setminus P)*(G\setminus P)=G.
\end{equation}

Recall $\theta$ from \eqref{eq:theta} and recall $ \E_1(\P) $ with base $E_1^\Sigma(\P)$ from \eqref{eq:m1}.  
Extend  $\theta\restriction  E(\P)$  to a map $\theta':  E_1(\P)\rightarrow \wp((3\times G)\times(3\times G))$, by letting 
\[ \bar a_{i, j}^{\theta'}=\bigcup_{g\in G\setminus\set a} g_{i, j}^\theta.\]
and extend $\theta'$ to $  E_1(\P)^\Sigma$ by $(\Sigma_{i\leq j<3} x_{ij})^{\theta^+}=\bigcup_{i\leq j<3} x_{ij}^{\theta'}$, for appropriate sums.  Then $\theta^+$  is a [finite] $\set{-, \comp }$-representation of $ \E_1(\P)$.

Conversely, if $\E_1(\P)$ is [finitely] $\set{-, \comp}$-representable then ${\bf P}$ embeds into a [finite] group, by Lemma~\ref{lem:zz}.   Hence the map ${\bf P}\mapsto   \E_1(\P)$ is a recursive reduction. Since the [finite] square partial group embedding problem is undecidable \cite{hirjac}, the theorem follows.
%
%
%The cases  (1) implies (2), (3), (4) are just the corresponding implications in Theorem \ref{thm:hirjac}.  For the converse direction, in each case, Lemma \ref{lem:5} shows that there is an $\mathscr{H}$-embedding of the square partial group ${\bf P}$, so that ${\bf P}$ embeds into a group (which will be finite, if the finiteness condition holds) by Lemma \ref{lem:sap}.
\end{proof}

Also recall the second main theorem.
\thmmainkernel*

\begin{proof}
Let ${\bf P}$ be a square partial group.  For signatures including negation but not meet we use  $ \E_1({\bf P})$, for signatures including meet but not complement we use $ \E_2({\bf P})$.
Both contain   $ \E_0(\P)$ as a  $\{\KL,\KR,\comp\}$ subreduct. 
Suppose $\P$ embeds into a [finite] group $(G, *)$.  The  restriction of $\theta$ (see \eqref{eq:theta}) to $  E(\P)$ also respects $\KL, \KR$, since all elements are injective, by \eqref{eq:domker}. The restriction of $\theta$ to $ \overline E(\P)$ also respects $\KL, \KR$ by \eqref{eq:negker}, hence these operators are respected over $ E_1(\P)$. Conversely, if $  E({\bf P})$ has been embedded as a $\{\KL,\KR,\comp\}$ semigroup into the unary semigroup of binary relations where $\KL(x)=x\comp x^\conv$ and $\KR(x)=x^\conv\comp x$, then  ${\bf P}$ embeds into a [finite] group, by  Theorem~\ref{thm:mainbiunary}, as required.
\end{proof}

%\begin{thm}\label{thm:main}
%Let $\tau$ be a signature containing $\{\comp,\KL,\KR\}$ and contained within either $\{\comp,\cdot,\join,\DD,\RR,\Ad,\KL,\KR,\leq,0,1,1'\}$ or $\{\comp,-,\join,\DD,\RR,\Ad,\KL,\KR,\leq,0,1,1'\}$.  Then the representability problem and the finite representability problem is undecidable for the signature $\tau$.
%\end{thm}

\begin{remark}
Our arguments will not work if $+$ or ${}^\conv$ is included in the signature.  For example, if $a,b\in P$ then $\KL(a_{13}+b_{13})$ should equal 
\[
(a_{13}+b_{13})\comp (a_{13}+b_{13})^\conv= a_{13}\comp a_{13}^\conv+a_{13}\comp b_{13}^\conv+b_{13}\comp a_{13}^\conv+b_{13}\comp b_{13}^\conv
\]
and we have no a priori knowledge of what the elements $a_{13}\comp b_{13}^\conv$ and $b_{13}\comp a_{13}^\conv$ should look like. In many such cases however there are other reasons to achieve undecidability: converse-free signatures extending $\{+,\cdot,\comp\}$ and any signature extending $\{\cdot,{}^\smile,\comp\}$ are already known to have undecidable representability problems (see~\cite{neu} and \cite{hirjac} respectively).  
There are no known negative results on the undecidability of subsignatures of $\{0,1,+,\leq,1',{}^\conv,\DD,\RR,\Ad,\KL,\KR,\comp\}$, though many are known to have no finite axiomatisation: see Andr\'eka and Mikul\'as for example \cite{andmik}.  
An exception is the oasis signature $\{\leq,\DD,\RR,{}^\conv,\comp\}$ (including with $1'$) which as noted in the introduction has a simple finite axiomatisation due to Bredikhin \cite{bre}; see also Hirsch and Mikul\'as~\cite{hirmik} where the FRP is additionally established.  
\end{remark}
Finally recall the third main theorem.
\thmchain*
\begin{proof} Let $n\geq 1$.  The proof of  Theorem \ref{thm:mainbiunary} goes through with $\KLn{n}, \KRn{n}$ for $\KL, \KR$, (using terms 
\[
\KLn{n}(x)= \stackrel{n\, \mbox{\tiny times}}{\overbrace{(x*x')*\ldots *(x*x')}}\ \text{ and }\ \KRn{n}(x)= \stackrel{n\, \mbox{\tiny times}}{\overbrace{(x'*x)*\ldots *(x'*x)}}).
\] so each signature 
$\{\KLn{n},\KRn{n},\comp\}$ has UDR and UDFR: this required only embeddability as subreduct of a unary semigroup, and binary relations form a unary semigroup.  Here we show that with respect to representability as binary relations, every term that can be expressed in the signature $\{\KLn{n},\KRn{n},\comp\}$ for infinitely many $n$ is expressible in the signature~$\{\comp\}$.  In order to track the definability of terms,  each  $\{\KLn{n},\KRn{n},\comp\}$-definable term  $t$ is equivalent to a $\set{ {}^\conv,\comp}$-term  and we let the \emph{size} of $t$ be the length of the shortest equivalent $\set{ {}^\conv,\comp}$-term.
   In Proposition~3 of~\cite{sch}, Schein showed that the free involuted semigroups are representable as algebras of binary relations: in other words, the variety generated by the representable $\{{}^\conv,\comp\}$ algebras is  the variety of all involuted semigroups.  Using the involuted semigroup axioms $(x\comp y)^\conv\approx y^\conv\comp x^\conv$ and $(x^\conv)^\conv\approx x$, it is easy to see that every term $t(x_1,\dots,x_n)$ in the language $\{{}^\conv,\comp\}$ can be written as a $\{\comp\}$-term in terms of $\{x_1,\dots,x_n,x_1^\conv,\dots,x_n^\conv\}$.  Moreover, this representation of terms is unique (up to associativity), because the involuted semigroup axioms do not enable any further identifications.  This is well known, but can be routinely proved by noting that the free semigroup on the alphabet $\{x_1,\dots,x_n,x_1^\conv,\dots,x_n^\conv\}$ becomes an involuted semigroup by defining a unary operation in accordance with the obvious left-to-right applications of the the laws $(x\comp y)^\conv\approx y^\conv\comp x^\conv$ and $(x^\conv)^\conv\approx x$.  
%To see why this is important, consider instead the corresponding question for representability as injective partial functions.  For injective partial functions we have $x\comp x^\conv\comp x\approx x$, so 
 % Thus we may  use the length of such terms, as $\{\comp\}$-terms in the alphabet $\{x_1,\dots,x_n,x_1^\conv,\dots,x_n^\conv\}$, as a kind of invariant measure of size.  
  As $\KLn{n}(x)$ and $\KRn{n}(x)$ (for any $n$) are terms in the involuted semigroup signature, of length $2n$, it follows that  any term $t$ in the signature $\{\comp,\KLn{n},\KRn{n}\}$ that includes at least one $\KLn{n}$ or $\KRn{n}$, has size at least $2n$.    It follows immediately that if $t$ is expressible in $\{\KLn{n},\KRn{n},\comp\}$ for infinitely many $n$, then it is a $\{\comp\}$-term.  This is the claim in Theorem \ref{thm:chain}.  
  \end{proof}

%\begin{lem}\label{lem:leftid}
%If $\KL(\KL(r))=\KL(r)$ and $\KL(r);r=r$, then $\KL(r)$ is a local equivalence relation,  $r$ acts identically on each element of the same block of $\KL(r)$ and no element of $X$ is related to distinct blocks of $\KL(r)$ by $r$.
%\end{lem}
%\begin{proof}
%The local equivalence relation property is by Lemma \ref{lem:equiv}, while the final statement comes directly from the definition of $\KL$.  For the second statement, consider $x_1,x_2$ lying in the same block of $\KL(r)$ and $(x_1,y)\in r$.  Now as $(x_2,x_1)\in \KL(r)$ we have $(x_2,y)\in r$ so that $r$ acts identically on all elements of the block containing $x_1$.
%\end{proof}
%\begin{lem}\label{lem:biunary}
%Let $a$ be an element of a bi-unary semigroup $\langle S,\comp,\KL,\KR\rangle$  for which $\KL(a)\comp a=a=a\comp\KR(a)$.  Then in any embedding of ${\bf S}$ into an involuted semigroup $\langle T,\comp,{}^\conv\rangle$ in which $\KL$ and $\KR$ are the usual derived operations $\KL(x)=x\comp x^\conv$ and $\KL(x)=x^\conv\comp x$, we have $\KL(a)\mathrel{\mathscr{R}}a$ and $\KR(a)\mathrel{\mathscr{L}}a$.
%\end{lem}
%\begin{proof}
%In keeping with the relation algebraic directions of this lemma, we let $\conv$ denote the involution.  Now $\KL(a)=a\comp a^\conv$ so that $\KL(a)\leq_\mathscr{R} a$.  But we assumed $\KL(a)\comp a=a$, so that $a\leq_\mathscr{R}\KL(a)$ as well.  The case of $\KR(a)\mathrel{\mathscr{L}}a$ is by symmetry.
%\end{proof}

\section{Appendix: the Kernel Problem}

Every binary relation $r$ has a kernel $\KL(r)$  defined by
\[
\KL:=\{(x,y)\in X\times X\mid  \exists z (x,z)\in r\And (y,z)\in r\}=r\comp r^\smile
\] 
and cokernel  $\KR(r):=r^\smile\comp r$.  When $r$ is a partial map (a function from some subset of $X$ into $X$) then the kernel is well known to be an equivalence relation on the domain of $r$.  In general the kernel need not be an equivalence relation, and neither need a relation $r$ necessarily be a partial map  when $\KL(r)$ is an equivalence relation.  
These operations are natural enough in their own right and have been found to play a role in the context of \emph{partition monoids}; see Proposition~4.2 in~\cite{easgra} for example. They also make an interesting parallel to the class of \emph{ample semigroups}, (originally introduced as type~A semigroups by Fountain~\cite{fou}).  Ample semigroups may be thought of as a $\{\DD,\RR,\comp\}$ approximation to the abstract models of algebras of injective partial maps in the signature $\{{}^\conv,\comp\}$.  
If we restrict our binary relations to the class of injective partial maps on $X$, then the signature $\{{}^\conv,\comp\}$ corresponds to the class of inverse semigroups by way of the Wagner-Preston representation~\cite{clipre},  the operations $\DD$, $\KL$ coincide,  and $\RR$, $\KR$ also coincide.  
Term reducts to the signature $\{\DD,\RR,\comp\}$, or equivalently $\{\KL,\KR,\comp\}$, are examples of ample semigroups, though not every ample semigroup arises in this way.  
Indeed, the representability of algebras in the signature $\{\DD,\RR,\comp\}\equiv \{\KL,\KR,\comp\}$ \emph{as injective partial maps} is known to be undecidable (Gould and Kambites \cite{goukam}; see \cite[Theorem 6.2]{jacsto:BS}).   

We consider the problem of recognising when a relation on a finite set is actually the kernel of some other relation on that set.  The following two lemmas were observed by James East.
\begin{lem}\label{lem:equiv}
A relation is fixed by $\KL$ if and only if it is a local equivalence relation, and likewise for $\KR$.
\end{lem}
\begin{proof}
The right to left implication is trivial.  Assume now that $r$ is a relation with $\KL(r)=r$.  The definition of $\KL$ ensures that $r$ is reflexive on its domain, and symmetric.  We need to prove transitivity.  Assume $(x,y),(y,z)\in r$.  As $r$ is symmetric we have $(z,y)\in r$ also, and then as $(x,y)$ and $(z,y)\in r$ we have $(x,z)\in \KL(r)=r$ as required.
\end{proof}

%We say that the \emph{irreflexive interior} of a relation $r$ is the relation obtained from $r$ by removing all loops. \nb{Do we need irref int?} If $r$ is a symmetric relation, then its irreflexive interior is a simple graph.  
The following two facts do not play a role in the rest of the paper, but are of interest in their own right. 
\begin{lem}\label{lem:graph}
A relation $r\in \Rel(X)$ is the kernel of some relation  if and only if it is is symmetric, locally reflexive and, the binary relation $r$ has an edge covering by at most $|X|$-many cliques \up(ignoring loops\up). %\nb{Took out irrefl interior}
\end{lem}
\begin{proof}
Suppose $r$ is the kernel of some $s$, so $r=s\comp s^\conv$ in $\Rel(X)$. Clearly $r$ is symmetric and reflexive over its domain.  Also, for each point $y$ in the range of $s$, the set $\{x\in X\mid  (x,y)\in s\}$ is a clique  of $r$, and these cliques cover $r$, so that the stated condition holds.  It is sufficient because  if it is satisfied by $r$ then we may select any injective function $\nu$ from the maximal cliques of  $r$ into $X$ and define a relation $s$ in $\Rel(X)$ by letting $(x,y)\in s$ whenever $x\in \nu^{-1}(y)$.
\end{proof}
An instance $(X, r)$ of the \emph{kernel problem} consists of a finite set $X$ and a binary relation $r\subseteq X\times X$, it is a yes-instance if $r$ is the kernel of some $s\in\Rel(X)$ and a no-instance otherwise.
\begin{cor}
The  kernel problem is \texttt{NP}-complete.
\end{cor}
\begin{proof}
It is clear that the {kernel problem} may be solved nondeterministically by guessing a relation $s$ and checking if $r=s\comp s^\conv\in\Rel(X)$.  
We prove NP-hardness by providing a reduction from the Clique Edge Colour Problem (CEC, also known as the Intersection Number problem), in which we are given a finite graph $G=(V,E)$ and a number $k\leq |E|$ and wish to know if there is a covering of the edges $E$ by at most $k$ many complete subgraphs (cliques) of $G$. 
The CEC problem is known to be \texttt{NP}-complete~\cite{orl}.  
So, given an instance $(V, E, k)$ of CEC, first check whether $k\geq|V|$ or not.  
If $k\geq |V|$ we reduce $(V, E, k)$ to $(V', E')$ where $V'\supset V$ is a set with $k+1$ vertices and $E'=E\cup \operatorname{Id}_{(V'\setminus V)}\cup (V'\setminus V)\times (V'\setminus V))$, in other words $(V', E')$ is the disjoint union of $(V, E)$ and a single clique  with ${k-|V|+1}$ nodes.  
Then $(V, E, k)$ is a yes-instance of CEC iff  there is a covering of $E$ by at most $k$ many cliques if and only if $(V', E')$ has an edge covering by at most $k+1$ cliques iff $(V', E')$ is a yes-instance of the kernel problem, for $k\geq |V|$.  
Otherwise $k<|V|$ and  let $(V', E')$ be the disjoint union of $(V, E)$ with the complete  graph $K_{2, |V|-k+2}$ bipartite except there are loops at all nodes,  so $|V'|= |V|+(|V|-k+4)$.  
Each edge between distinct nodes  of $K_{2, |V|-k+2}$ is a maximal clique, by bipartiteness of the irreflexive 
part, so the smallest number of cliques that cover the edges of  this part is $2(|V|-k+2)$.  
We reduce $(V, E, k)$ to $(V', E')$
and then $(V, E, k)$ is a yes-instance of $CEC$ iff $(V', E')$ has 
a covering by at most $k+2(|V|-k+2)=|V'|$ cliques iff $(V', E')$ is a yes-instance of the kernel problem, so the reduction is correct when $k<|V|$ too.\end{proof}

\bibliographystyle{amsplain}

%%%%%%%%%%%%%%%%%%%%%%%%%%%%%%%%%%%%%%%%%%%%%%%%%%%%%%%%%%%%%%%%%%
%%%%%%%%%%%%%%%%%%%%%%%%%%%%%%%%%%%%%%%%%%%%%%%%%%%%%%%%%%%%%%%%%%
\end{document}